\documentclass[10pt,reqno]{amsart}
\usepackage{amscd,amssymb,amsthm}
\usepackage{graphicx}

\usepackage{enumerate}
\theoremstyle{plain}
\newtheorem{theorem}{Theorem}[section]
\newtheorem{corollary}[theorem]{Corollary}

\newtheorem{example}[theorem]{Example}

\theoremstyle{remark}

\newcommand{\reel}{\mathbb{R}}

\newcommand{\ent}{\mathbb{Z}}
\newcommand{\comp}{\mathbb{C}}
\newcommand{\wa}{{\rm 1\mkern-4.7mu}{\rm I}}
\newcommand{\tore}{\mathbb{T}}

\newcommand{\eps}{\varepsilon}
\newcommand{\abs}[1]{\left\vert #1\right\vert }
\newcommand{\adh}[1]{\overline{#1} }
\newcommand{\N}[1]{{\left\Vert#1\right\Vert}}
\newcommand{\bg}{\medskip\goodbreak}

\newcommand{\egdef}{\buildrel {\scriptscriptstyle{\rm def}}\over{=}}

\newcommand{\Log}{\operatorname{Log}}

\begin{document}
\title[A Mixte Parseval-Plancherel Formula]
{A Mixte Parseval-Plancherel Formula}
\author[Omran Kouba]{Omran Kouba$^\dag$}
\address{Department of Mathematics \\
Higher Institute for Applied Sciences and Technology\\
P.O. Box 31983, Damascus, Syria.}
\email{omran\_kouba@hiast.edu.sy}
\keywords{Fourier series, Fourier Transform, Oscillatory integrals.}
\subjclass[2010]{42A16, 42A38, 42B20.}
\thanks{$^\dag$ Department of Mathematics, Higher Institute for Applied Sciences and Technology.}


\date{\today}
\begin{abstract}
In this note, a general formula is proved.
 It expresses the integral  on the line
of the product of a  function $f$ and a periodic function $g$ in terms of
the Fourier transform of $f$ and the Fourier coefficients of $g$. This allows the evaluation of some
oscillatory integrals.\par
\end{abstract}
\smallskip\goodbreak

\maketitle
\section{\sc Introduction and Notation}\label{sec1}
\bg

In \cite{Wal} the following integral  was described  as ``difficult'':
\begin{equation}\label{E11}
\int_{-\infty}^{\infty}\frac{dx}{\big(\cosh a+\cos x\big)\cosh x } \quad\text{for $a>0$,}
\end{equation}
it was used to test the trapezoidal rule after transforming the integral using a ``sinh'' transformation. Also, in
\cite{asy} S. Tsipelis proposed to evaluate the following integral 
\begin{equation}\label{E12}
\int_{-\infty}^{\infty}\frac{\log(\cos^2x)}{1+e^{2\abs{x}}}\,dx.
\end{equation}

Both integrals are of the form $\int_\reel f(x)g(x)dx$ where  $g$ is a $2\pi$-periodic function. The particular case, where
$f$ is of the form $x\mapsto1/(x+z)$, (for some $z\in\comp\setminus\reel$,) was thoroughly investigated in \cite{kou}  using methods
that are different from those discussed in this paper.

\bg
In this note,  we prove a general formula, that
allows us to express this kind of integrals in terms of  the Fourier transform of $f$
 and the Fourier coefficients of $g$.

\bg
Before we proceed, let us recall some standard notation. The spaces $L^1(\reel)$, $L^2(\reel)$, and 
 $L^{2,\text{loc}}(\reel)$ are, respectively, the space of integrable functions,
the space of square integrable functions, and
the space of locally square integrable functions on $\reel$. The spaces $L^1(\reel)$ and $L^2(\reel)$ 
are equipped with the standard norms denoted $\N{\cdot}_1$ and $\N{\cdot}_2$:
\[
 \N{f}_{p} =\left( \int_{\reel}\abs{f(t)}^p\,dt\right)^{1/p},\quad\text{for
 $p=1,2$.}
\]

\bg

We consider also $L^1(\tore)$, (resp. $L^2(\tore)$), the space of integrable, (resp. square integrable),
$2\pi$-periodic functions. The spaces $L^1(\tore)$ and $L^2(\tore)$ are equiped with the
 standard norms denoted $\N{\cdot}_{L^1(\tore)}$ and $\N{\cdot}_{L^2(\tore)}$ defined as folows:
\[
 \N{f}_{L^p(\tore)} =\left(\frac{1}{2\pi}\int_{\tore}\abs{f(t)}^p\,dt\right)^{1/p},\quad\text{for
 $p=1,2$.}
\]

\bg
For a function $f\in L^1(\reel)$ we recall that its Fourier transform $\widehat{f}$ is defined 
by
\[
\widehat{f}(\omega)=\int_{\reel}f(t)e^{-i\omega t}\,dt,\quad\text{for $\omega\in\reel$.}
\]

\bg
And for a $2\pi$-periodic  function $g\in L^1(\tore)$ we recall that 
the exponential Fourier coefficient  $ C_n(g)$ of $g$ is defined 
by
\[
C_n(g)=\frac{1}{2\pi}\int_{\tore}g(t)e^{-i n t}\,dt,\quad\text{for $n\in\ent$,}
\]
or more generally $C_n(g)=\frac{1}{T}\int_0^Tg(x)e^{-2\pi i n x/T}\,dx$, if $g$ is $T$-periodic.
\bg
In section \ref{sec2} we will prove our main results 
and in section \ref{sec3} we will give some detailed examples and applications.
\bg

\bg

\section{\sc The Main Result}\label{sec2}
\nobreak
In this section we state and prove the main theorem.\bg
\begin{theorem}[The mixed Parseval-Plancherel formula]\label{th1}
Consider a function $f$ from $L^{2,{\rm loc}}(\reel)$, and a $2\pi$-periodic function $g$ from $L^{2}(\tore)$. Suppose that
\begin{equation}\label{E:H}
M(f)\egdef\sum_{k\in\ent}\N{\wa_{I_k}f}_2<+\infty,
\end{equation}
where $\wa_{I_k}$ is the characteristic function of the interval ${I_k=[2\pi k,2\pi(k+1)]}$. Then
\begin{equation}\label{E:C}
\int_{\reel}f(x)\adh{g(x)}\,dx =\sum_{n\in\ent}\widehat{f}(n)\adh{C_n(g)}.
\end{equation}
where $\widehat{f}$ is the Fourier transform of $f$, and $(C_n(g))_{n\in\ent}$ is the family of exponential Fourier coefficients of $g$.
\end{theorem}
\begin{proof} First, note that $\N{\wa_{I_k}f}_1\leq\sqrt{2\pi}\,\N{\wa_{I_k}f}_2$ for every $k\in\ent$. It follows that
\[
\int_\reel\abs{f(x)}\,dx=\sum_{k\in\ent}\N{\wa_{I_k}f}_1\leq\sqrt{2\pi}\, M(f)<+\infty.
\]
Thus, $f$ belongs to $L^1(\reel)$, and we can consider its Fourier transform. Similarly,
\[
\int_\reel\abs{f(x)g(x)}\,dx=\sum_{k\in\ent}\int_{I_k}\abs{f(x)g(x)}\,dx\leq
\sqrt{2\pi}\,M(f)\,\N{g}_{L^2(T)}<+\infty,
\]
and consequently $fg$ belongs also to $L^1(\reel)$.

Now, let us consider the the family of functions $(f_k)_{k\in\ent}$ defined by $f_k(x)=f(x+2\pi k)$.
Clearly $\N{\wa_{m}f_k}_2=\N{\wa_{m+k}f}_2$. Thus
\[
\sum_{k\in\ent}\N{\wa_{m}f_k}_2=M(f)<+\infty
\]
and the series $\sum_{k\in\ent} \wa_{m}f_k$ is normally convergent in $L^2(\reel)$ for every $m\in\ent$. This proves that
the formula $F=\sum_{k\in\ent}f_k $
defines a function $F$ that belongs to $L^{2,{\rm loc}}(\reel)$. Moreover, this function is clearly $2\pi$-periodic, and
$\N{F}_{L^2(\tore)}\leq\frac{1}{\sqrt{2\pi}}\,M(f)$. Now, the classical Parseval's formula, (see \cite[Chap. I, \S 5.]{kats} or
\cite[Chap. 5, \S 3.]{tol},) proves that
\begin{equation}\label{E:th11}
\frac{1}{2\pi}\int_{\tore}F(x)\adh{g(x)}\,dx=\sum_{n\in\ent}C_n(F)\adh{C_n(g)}.
\end{equation}
Using the fact that $\sum_{k=-n}^{n-1}\wa_{I_0}f_k$ converges to $\wa_{I_0}F$ in $L^2(\reel)$, and that
$\wa_{I_0}g\in L^2(\reel)$,  we conclude that
\begin{align}\label{E:th12}
\int_{0}^{2\pi}F(x)\adh{g(x)}\,dx&=\lim_{n\to\infty}\sum_{k=-n}^{ n-1}\int_{0}^{2\pi}f_k(x)\adh{g(x)}\,dx\notag\\
&=\lim_{n\to\infty}\sum_{k=-n}^{n-1}\int_{2\pi k}^{2\pi (k+1)}f(x)\adh{g(x)}\,dx\notag\\
&=\int_{\reel}f(x)\adh{g(x)}\,dx
\end{align}
where,  for the last equality, we used the fact that  $fg\in L^1(\reel)$. 

\bg
Similarly,
\begin{align}\label{E:th13}
2\pi C_n(F)&=\int_{0}^{2\pi}F(x)e^{-int}\,dx =\lim_{n\to\infty}\sum_{k=-n}^{n-1}\int_{0}^{2\pi}f_k(x)e^{-int}\,dx\notag\\
&=\lim_{n\to\infty}\sum_{k=-n}^{n-1}\int_{2\pi k}^{2\pi (k+1)}f(x)e^{-int}\,dx\notag\\
&=\int_{\reel}f(x)e^{-int}\,dx=\widehat{f}(n)
\end{align}
where,  again, we used the fact that  $f\in L^1(\reel)$ for the last equality. Replacing \eqref{E:th12} and \eqref{E:th13} in \eqref{E:th11}, the desired
formula follows.
\end{proof}

\bg
The next corollary is straightforward.
\bg
\begin{corollary}\label{cor1}
Consider a function $f$ from $L^{2,{\rm loc}}(\reel)$, and a $T$-periodic, square integrable function $g$. Suppose that
\begin{equation}\label{E:H1}
M_T(f)\egdef\sum_{k\in\ent}\N{\wa_{[k T,(k+1)T]}f}_2<+\infty,
\end{equation}
Then
\begin{equation}\label{E:C1}
\int_{\reel}f(x)\adh{g(x)}\,dx =\sum_{n\in\ent}\widehat{f}\left(\frac{2\pi n}{T}\right)\adh{C_n(g)}.
\end{equation}
where $\widehat{f}$ is the Fourier transform of $f$, and $(C_n(g))_{n\in\ent} $ is the  family of exponential Fourier coefficients of $g$.
\end{corollary}

\bg
\section{\sc Examples}\label{sec3}
\nobreak
\begin{example} \end{example}

For positive real numbers $a$ and $b$, let $g$ and $f$ be the functions defined by
\[
g(x)=\frac{1}{\cosh a+\cos x},\qquad f(x)=\frac{1}{\cosh(b x)},
\]
It is known \cite[Chap.I, \S 9]{erd} that $\widehat{f}(\omega)=\frac{\pi}{b}f\left(\frac{\pi}{2b}\omega\right)$. Moreover,
it is easy to note that for every $k\in\ent$ we have $\N{\wa_{I_k}f}_2\leq B e^{-2\pi b\abs{k}}$ for some absolute constant $B$.

\bg
On the other hand,
it is easy to check that
\[
g(x)=\frac{1}{\sinh a}\sum_{n\in\ent}(-1)^ne^{-\abs{n}a}e^{i n x},
\]
that is 
\[
C_n(g)=\frac{(-1)^ne^{-\abs{n}a}}{\sinh a},\quad\text{for $n\in \ent$.}
\]
Hence, using Theorem \ref{th1}, we obtain
\[
 \int_{-\infty}^\infty\frac{dx}{(\cosh a+\cos x)\cosh(bx)}=\frac{\pi}{b\sinh a }
+\frac{2\pi}{b\sinh a }\sum_{n=1}^\infty\frac{(-1)^ne^{- n a}}{\cosh(\pi n/(2 b))}
\]
In particular, for $b=1$, we obtain the following expression of the integral \eqref{E11} as a rapidly convergent series: 
\[
 \int_{-\infty}^\infty\frac{dx}{(\cosh a+\cos x)\cosh x}=\frac{\pi}{\sinh a }
+\frac{2\pi}{\sinh a }\sum_{n=1}^\infty\frac{(-1)^ne^{- n a}}{\cosh(\pi n/2)}.
\]
This is a simpler alternative series expansion to the one obtained in \cite{Wal}.
\bg
\begin{example} \end{example}

In our second example, let $g$ and $f$ be the functions defined by
\[
g(x)=\log(\cos^2x),\qquad f(x)=\frac{1}{1+e^{2\abs{x}}}.
\]
It is easy to note that for every $k\in\ent$ we have $\N{\wa_{I_k}f}_2\leq B e^{-2\pi \abs{k}}$ for some absolute constant $B$. Moreover,
\begin{align*}
\widehat{f}(\omega)&=2\int_{0}^{\infty}\frac{e^{-2x}}{1+e^{-2x}}\cos(\omega x)\,dx\\
&=2\sum_{k=1}^\infty(-1)^{k-1}\int_0^\infty e^{-2kx}\cos(\omega x)\,dx\\
&=\sum_{k=1}^\infty(-1)^{k-1}\frac{4k}{4k^2+\omega^2}.
\end{align*}
On the other hand, since 
\[g(x)=2\log\abs{1+e^{2ix}} -2\log2=2\Re \Log(1+e^{2ix})-2\log2\]
with $\Log$ being the principal branch of the logarithm, we conclude that for every $n\in\ent$ we have
\[C_{2n+1}(g)=0,\quad\text{and}\quad
C_{2n}(g)=\left\{\begin{matrix}
(-1)^{n-1}/\abs{n}&\text{if}&n\ne0\\
-2\log2&\text{if}&n=0
\end{matrix}\right.
\]
\bg

Using Theorem \ref{th1}, we obtain
\begin{align}\label{ex1}
 \int_{-\infty}^\infty\frac{\log(\cos^2x)}{1+e^{2\abs{x}}}\,dx&=
\sum_{n\in\ent}\widehat{f}(2n)\adh{C_{2n}(g)}\notag\\
&=-2\log2\sum_{k=1}^\infty\frac{(-1)^{k-1}}{k}+2\sum_{n=1}^\infty\left(\sum_{k=1}^\infty(-1)^{k+n}\frac{k}{(k^2+n^2)n}\right)\notag\\
&=-2\log^22+2J
\end{align}
with
\begin{equation}\label{ex2}
J=\sum_{n=1}^\infty\left(\sum_{k=1}^\infty(-1)^{k+n}\frac{k}{(k^2+n^2)n}\right)
\end{equation}
Now, this double series is not absolutely convergent, so we must be carful. First, exchanging the roles of $k$ and $n$ we have
\[
J=\sum_{k=1}^\infty\left(\sum_{n=1}^\infty\frac{(-1)^{k+n}n}{p(n^2+k^2)}\right)
\]
Now, using the properties of convergent alternating we have
\[\sum_{n=1}^\infty\frac{(-1)^{k+n}n}{k(n^2+k^2)}=
\sum_{n=1}^{q-1}\frac{(-1)^{k+n}n}{k(n^2+k^2)}+R_q(k),
\]
with
\[
R_q(k)=\frac{(-1)^k}{k}\sum_{n=q}^\infty\frac{(-1)^{n}n}{ n^2+k^2}
\quad\text{and}\quad \abs{R_q(k)}\leq\frac{1}{k}\cdot\frac{q}{k^2+q^2}
\]
Thus
\[
J=\sum_{n=1}^{q-1}\left(\sum_{k=1}^\infty\frac{(-1)^{k+n}n}{k(n^2+k^2)}\right)+\eps_q
\]
with $\eps_q=\sum_{k=1}^\infty R_q(k)$. But
\[
\eps_q\leq \sum_{k=1}^\infty \frac{q}{k(k^2+q^2)}
\]
Now, since $\frac{q}{k(k^2+q^2)}\leq \frac{1}{2k^2}$ for every $q$,  the series $\sum 1/(2k^{2})<+\infty$ and 
$\lim_{q\to\infty}\frac{q}{k(k^2+q^2)}=0$ for every $k$, we conclude that $\lim_{q\to\infty}\eps_q=0$, So, letting 
$q$ tend to $+\infty$ we conclude that
\begin{equation}\label{ex3}
J=\sum_{n=1}^{\infty}\left(\sum_{k=1}^\infty\frac{(-1)^{k+n}n}{k(n^2+k^2)}\right)
\end{equation}
Taking the  sum of the two expressions \eqref{ex2} and \eqref{ex3} of $J$ we obtain 
\[
2J=\sum_{n=1}^{\infty}\left(\sum_{k=1}^\infty\frac{(-1)^{k+n}}{ n^2+k^2}\left(\frac{n}{k}+\frac{k}{n}\right)\right)=
\sum_{n=1}^{\infty}\left(\sum_{k=1}^\infty\frac{(-1)^{k+n}}{ nk}\right)=(-\log2)^2=\log^22.
\]
Replacing back in \eqref{ex1} we obtain
\[
 \int_{-\infty}^\infty\frac{\log(\cos^2x)}{1+e^{2\abs{x}}}\,dx=- \log^22.
\]

\begin{example} \end{example}

For positive real numbers $a$ and $b$, let $g$ and $f$ be the functions defined by
\[
g(x)=\frac{1}{\cosh a-\cos x},\qquad f(x)= e^{-x^2/(4b)},
\]
It is known \cite[Chap.I, \S 4]{erd} that $\widehat{f}(\omega)=2\sqrt{\pi b}f\left(2 b\omega\right)$.
 Moreover,
\[
C_n(g)=\frac{e^{-\abs{n}a}}{\sinh a},\quad\text{for $n\in \ent$.}
\]
Hence
\[
\int_\reel\frac{e^{-x^2/(4b)}}{\cosh a-\cos x}\,dx
=\frac{2\sqrt{\pi b}}{\sinh a}\left(1
+
2\sum_{n=1}^\infty e^{-a n-b n^2}\right)
\]
In particular, for $b=a$ we get
\[
\int_\reel\frac{e^{-x^2/(4a)}}{\cosh a-\cos x}\,dx
=\frac{2\sqrt{\pi a}}{\sinh a}\left(1+2\sum_{n=1}^\infty e^{-an(n+1)}\right)
=\frac{2\sqrt{\pi a}}{\sinh a}(e^{a/4}\vartheta_2(0,e^{-a})-1),
\]
where $\vartheta_2(u,q)$ is one of the well-known Jacobi Theta functions \cite{Wat}[Chap. XXI].

\end{document}